\documentclass[a4paper,10pt]{amsart}

\usepackage[T1]{fontenc}
\usepackage{quiver}
\usepackage{tikz-cd}
\usepackage{amsfonts}
\usepackage{amssymb}
\usepackage{amsthm}
\usepackage{amsmath}
\usepackage{hyperref}
\usepackage{geometry}

\theoremstyle{plain}
\newtheorem{theorem}{Theorem}[section]

\theoremstyle{plain}
\newtheorem{corollary}[theorem]{Corollary}

\theoremstyle{plain}
\newtheorem{lemma}[theorem]{Lemma}

\theoremstyle{plain}
\newtheorem{proposition}[theorem]{Proposition}

\theoremstyle{definition}
\newtheorem{definition}[theorem]{Definition}

\theoremstyle{definition}
\newtheorem{example}[theorem]{Example}

\theoremstyle{remark}
\newtheorem{remark}[theorem]{Remark}

\begin{document}

\title{Models for rational $(\infty, 1)$-categories}
\author{Eleftherios Chatzitheodoridis}
\address{Department of Mathematics, University of Virginia, Charlottesville, VA, USA}
\email{thp5uc@virginia.edu}
\maketitle

\begin{abstract}
We introduce rational $(\infty, 1)$-categories, which are $(\infty, 1)$-categories enriched in spaces whose higher homotopy groups are rational vector spaces. We provide two models for rational $(\infty, 1)$-categories, rational complete Segal spaces and rational Segal categories, and we show that they are equivalent.
\end{abstract}

\section{Introduction}\label{section1}
The study of $(\infty, 1)$-categories is intended to understand categories with a notion of homotopy between their morphisms. Namely, an $(\infty, 1)$-category is a category enriched in spaces, possibly weakly. This enrichment yields $k$-morphisms for all $k$ given by homotopies. However, all morphisms in dimension higher than $1$ are invertible up to homotopy because homotopies can be traveled in reverse.

Our understanding of $(\infty, 1)$-categories has been advanced thanks to the development of various models for $(\infty, 1)$-categories, that is, mathematical objects that exhibit the structure of an $(\infty, 1)$-category. More specifically, Quillen's concept of a model category from \cite{Quillen1967} constitutes a natural and convenient framework for the purpose of modeling $(\infty, 1)$-categories. Conceptually, a model structure on a category packages a homotopy theory on that category, for it allows us to define homotopies, function complexes, and an associated homotopy category. In this framework, one not only introduces mathematical objects that encode the structure of an $(\infty, 1)$-category, but also develops model categories that describe the homotopy theory of $(\infty, 1)$-categories.

There is a wide variety of equivalent models for $(\infty, 1)$-categories; a treatment of those models and their equivalence is provided by Bergner in \cite{Bergner2018}. In this paper, we use two such models, both of which are particular kinds of simplicial spaces: complete Segal spaces and Segal categories. Complete Segal spaces, as defined by Rezk in \cite{Rezk2001}, are simplicial spaces satisfying a condition that encodes morphism composition up to homotopy. Rezk's work yields a cartesian model category whose fibrant objects are precisely the complete Segal spaces; this model category is produced as a left Bousfield localization of the Reedy model structure on simplicial spaces.

Segal categories are also simplicial spaces that satisfy the same condition and have a discrete space of $0$-simplices; the discreteness condition accounts for the fact that we encode a set of objects. Segal categories were introduced by Dwyer, Kan, and Smith in \cite{DKS1989}, albeit under a different name, and they were developed by Hirschowitz and Simpson in \cite{HS2001}, as well as by Pellissier in \cite{Pellissier2003} from the homotopical perspective. In \cite{Bergner2007}, Bergner develops two equivalent model categories for the purpose of studying Segal categories, and she establishes that both model categories are equivalent to Rezk's model category for complete Segal spaces.

Having such a rich theory at their disposal, one may be motivated to investigate $(\infty, 1)$-categories enriched in a class of spaces with a prescribed homotopical property; a goal of such an investigation can be to define equivalent analogs of the existing models for $(\infty, 1)$-categories. For instance, in the setting of equivariant homotopy theory, equivariant complete Segal spaces were defined by Bergner and Chadwick in \cite{BC2015}, and Bergner also produced equivalent models for $(\infty, 1)$-categories enriched in spaces with a group action in \cite{Bergner2017}.

In this paper, we provide analogs of complete Segal spaces and Segal categories for rational homotopy theory, which we then show to be equivalent. In order not to require mapping spaces to be simply connected, we adopt the framework for the rational homotopy theory of non-simply connected spaces due to G{\'{o}}mez-Tato, Halperin, and Tanr{\'e} \cite{GTHT2000}. In that framework, a space is said to be rational if its higher homotopy groups are rational vector spaces, with no conditions imposed on its fundamental group or its set of path components.

With this notion of a rational space, we provide two equivalent models for rational $(\infty, 1)$-categories, which we define to be $(\infty, 1)$-categories enriched in rational spaces. Our first model is that of rational complete Segal spaces, which are complete Segal spaces with rational mapping spaces, and our second model is that of rational Segal categories, which are Segal categories with rational mapping spaces. By a left Bousfield localization of Rezk's model category for complete Segal spaces and Bergner's model categories for Segal categories, we produce a model category for rational complete Segal spaces and two model categories for rational Segal categories, respectively. We also establish that our model category for rational complete Segal spaces is cartesian; to do so, we apply Rezk's criterion from \cite{Rezk2001} on when a localization of the Reedy model structure remains compatible with the cartesian closure of simplicial spaces. Moreover, the equivalence between complete Segal spaces and Segal categories descends to an equivalence between rational complete Segal spaces and rational Segal categories.

\begin{theorem}\label{bogthm} There exist:

\begin{enumerate}

\item a model category whose fibrant objects are the rational complete Segal spaces;

\item a model category whose fibrant objects are the Reedy fibrant rational Segal categories; and

\item a model category whose fibrant objects are the levelwise fibrant rational Segal categories.

\end{enumerate}
These model categories are Quillen equivalent.
\end{theorem}
To prove these results at the level of simplicial spaces, we first develop a model-categorical approach to the rational homotopy theory of non-simply connected spaces. Specifically, we left Bousfield localize the model structure on simplicial sets to obtain a model category whose fibrant objects are the rational spaces. For this localization, we establish a characterization of rational spaces in terms of function complexes.

In fact, our argument is not exclusive to rational spaces, and it works for any class of spaces that arises from a left Bousfield localization of the model structure on spaces. For example, in the chromatic homotopy-theoretic framework developed by Heuts in \cite{Heuts2021}, our work yields two equivalent models for $(\infty, 1)$-categories enriched in $v_{n}$-periodic spaces for a non-negative integer $n$, with the case $n=0$ being that of rational $(\infty, 1)$-categories.

In the process, we investigate the relationship between complete Segal spaces and Segal categories through the lens of the left transfer of a model category, as developed by Bayeh, Hess, Karpova, K{\k{e}}dziorek, Riehl, and Shipley in \cite{BHKKRS2015}. We show that the model category for Reedy fibrant Segal categories can be described as a left transfer of the model category for complete Segal spaces, as well as that our model category for Reedy fibrant rational Segal categories is left-induced from our model category for rational complete Segal spaces. The former left transfer might be of independent interest as it greatly simplifies Bergner's proof from \cite{Bergner2007}.

Lastly, we hope that the techniques in this paper help us develop rational homotopy-theoretic analogs of other models for $(\infty, 1)$-categories in the future. Such models are quasi-categories, as introduced by Boardman and Vogt in \cite{BV1973} and studied by Joyal in \cite{Joyal2002} and Lurie in \cite{Lurie2009}, as well as simplicial categories, as studied by Bergner in \cite{Bergner2007_3}. We are working on introducing rational analogs of these two models, and then showing that they are equivalent to the two equivalent models for rational $(\infty, 1)$-categories that we introduce in this paper.

\subsection*{Organization of the paper}
Section \ref{section2} comprises the necessary background in simplicial sets and model categories. In Section \ref{section3}, we establish our characterization of rational spaces in terms of function complexes, and we use it to give our model category for rational spaces. Section \ref{section4} has background on complete Segal spaces. Section \ref{section5} is on Segal categories, and it contains the description of the model category for Reedy fibrant Segal categories as a left transfer of the model category for complete Segal spaces. In Section \ref{section6}, we introduce rational Segal categories and we give two model categories for them. In Section \ref{section7}, we introduce rational complete Segal spaces and we produce a cartesian model category for them. Lastly, in Section \ref{section8}, we establish the equivalence of rational complete Segal spaces and rational Segal categories, and we describe our model category for Reedy fibrant rational Segal categories as a left transfer of our model category for rational complete Segal spaces.

\subsection*{Acknowledgments}
I am grateful to my PhD advisor at the University of Virginia, Julie Bergner, for her guidance, feedback, and support at all stages of the research presented in this paper. I would also like to thank Pedro Brunialti Lima de Andrade for helpful conversations about this work.

\section{Simplicial sets and model categories}\label{section2}
We first go over some elements of the theory of simplicial sets and model categories that we use in this paper.

\subsection{Simplicial sets}\label{section2.1} Simplicial sets are a combinatorial model for spaces; our reference is \cite[Chapter I]{GoerssJardine1999}.

\begin{definition}\label{defsimpcatdef} The \emph{simplex category} $\Delta$ is the category whose objects are the sets $[n]=\{0, 1, \dots, n\}$ for a non-negative integer $n$ and whose morphisms are the order-preserving maps $\theta \colon [n] \rightarrow [m]$.
\end{definition}
The simplex category $\Delta$ packages the combinatorial data in the definition of a simplicial set as follows.

\begin{definition}\label{defssetdef} A \emph{simplicial set} is a functor $X \colon \Delta^{\mathrm{op}} \rightarrow {\mathcal{S}\mathrm{ets}}$. In detail, a simplicial set comprises a \emph{set of $n$-simplices} $X_{n}$ for every non-negative integer $n$, \emph{face maps} $d_{i} \colon X_{n} \rightarrow X_{n-1}$ for $0 \leq i \leq n$, and \emph{degeneracy maps} $s_{i} \colon X_{n} \rightarrow X_{n+1}$ for $0 \leq i \leq n$ subject to a list of identities on how the faces and degeneracies are composed.
\end{definition}
We write ${\mathcal{SS}\mathrm{ets}}$ for the category of simplicial sets and their simplicial maps (natural transformations).

\begin{example}\label{defssetdef1} The \emph{standard $n$-simplex} is $\Delta[n]=\operatorname{Hom}_{\Delta}(-, [n])$; the identity $\operatorname{id}_{[n]}$ defines an $n$-simplex $\sigma_{n}$. The \emph{boundary} $\partial\Delta[n]$ is the smallest subcomplex of $\Delta[n]$ containing $d_{i}(\sigma_{n})$ for $0 \leq i \leq n$. For $n \geq 1$ and $0 \leq k \leq n$, the \emph{horn} $\Lambda[n, k]$ is the smallest subcomplex of $\Delta[n]$ containing $d_{i}(\sigma_{n})$ for $i \neq k$.
\end{example}

\begin{example}\label{defssetdef4} For a small category $\mathcal{C}$, its \emph{nerve} $\operatorname{nerve}(\mathcal{C})$ is the simplicial set whose set of $0$-simplices is the set of objects of $\mathcal{C}$ and whose the set of $n$-simplices for $n \geq 1$ is the set of strings of $n$ composable morphisms in $\mathcal{C}$.
\end{example}
We can geometrically realize a simplicial set $X$ to obtain a CW complex $|X|$. Writing ${\mathcal{T}\mathrm{op}}$ for the category of topological spaces, we denote the geometric realization functor by $|-| \colon {\mathcal{SS}\mathrm{ets}} \rightarrow {\mathcal{T}\mathrm{op}}$. Its right adjoint sends a topological space $Y$ to its \emph{singular complex} $\operatorname{Sing}(Y)$, which has $n$-simplices $\operatorname{Sing}(Y)_{n}=\operatorname{Hom}_{\mathcal{T}\mathrm{op}}(|\Delta[n]|, Y)$.

Our work also involves the following notion of a fibration for simplicial sets.

\begin{definition}\label{defkanfibrdef} A \emph{Kan fibration} is a simplicial map $p \colon X \rightarrow Y$ that has the right lifting property with respect to all horn inclusions.
A \emph{Kan complex} is a simplicial set $X$ such that the map $X \rightarrow \Delta[0]$ is a Kan fibration.
\end{definition}

\subsection{Model categories}\label{section2.2} We give a brief review of model categories from \cite[\S 3]{DwyerSpalinski1995} and \cite[Chapters 7-13]{Hirschhorn2003}.

\begin{definition}\label{defmodelcatdef} A \emph{model category} comprises a category $\mathcal{M}$ that has all finite limits and colimits, as well as three distinguished classes of morphisms in $\mathcal{M}$: a class of \emph{weak equivalences}, a class of \emph{cofibrations}, and a class of \emph{fibrations}. A weak equivalence that is also a cofibration is called an \emph{acyclic cofibration}. A weak equivalence that is also a fibration is called an \emph{acyclic fibration}. This data satisfies the axioms in \cite[Definition 3.3]{DwyerSpalinski1995}.
\end{definition}

\begin{remark}\label{defmodelcatdefrmk2} In a model category, the weak equivalences and the cofibrations together determine the fibrations.
\end{remark}

\begin{example}[{\cite[Theorem 7.10.12]{Hirschhorn2003}}]\label{defmodelcatdef1} The category ${\mathcal{SS}\mathrm{ets}}$ has a model structure in which a map $f$ is

\begin{enumerate}

\item a weak equivalence if and only if $|f|$ is a weak homotopy equivalence;

\item a cofibration if and only if $f$ is a monomorphism; and

\item a fibration if and only if $f$ is a Kan fibration.

\end{enumerate}

\end{example}
The model category ${\mathcal{SS}\mathrm{ets}}$ has a list of features that we rely on, but do not explicitly use in this paper; namely, ${\mathcal{SS}\mathrm{ets}}$ is combinatorial \cite{Dugger2001}, left proper \cite[Corollary 13.1.4]{Hirschhorn2003}, and simplicial \cite[Example 9.1.13]{Hirschhorn2003}. In particular, for two simplicial sets $X$ and $Y$, we have a function complex $\operatorname{Map}_{{\mathcal{SS}\mathrm{ets}}}(X, Y)$ with $n$-simplices \[\operatorname{Map}_{{\mathcal{SS}\mathrm{ets}}}(X, Y)_{n}= \operatorname{Hom}_{{\mathcal{SS}\mathrm{ets}}}(X \times \Delta[n], Y).\]

\begin{definition}\label{defmodelcatdefrmk1} Let $\mathcal{M}$ be a model category with initial object $\emptyset$ and terminal object $\ast$. An object $X$ of $\mathcal{M}$ is \emph{cofibrant} if the morphism $\emptyset \rightarrow X$ is a cofibration, and it is \emph{fibrant} if the unique morphism $X \rightarrow \ast$ is a fibration.
\end{definition}
Lastly, we need a notion of equivalence saying that two model categories encode the same homotopy theory.

\begin{definition}\label{defquilleneqdef} A \emph{Quillen adjunction} is an adjunction between model categories $F \colon \mathcal{M} \rightleftarrows \mathcal{N} \colon G$ such that $F$ preserves cofibrations and $G$ preserves fibrations. A \emph{Quillen equivalence} is a Quillen adjunction $F \colon \mathcal{M} \rightleftarrows \mathcal{N} \colon G$ such that, for every cofibrant object $X$ of $\mathcal{M}$ and every fibrant object $Y$ of $\mathcal{N}$, a morphism $F(X) \rightarrow Y$ is a weak equivalence in $\mathcal{N}$ if and only if its adjoint morphism $X \rightarrow G(Y)$ is a weak equivalence in $\mathcal{M}$.
\end{definition}
The following result renders the geometric realization-singular complex adjunction a Quillen equivalence.

\begin{theorem}[{\cite[Chapter I, Theorem 11.4]{GoerssJardine1999}}]\label{quilleneqsecretly} For every topological space $Y$, the map $|\operatorname{Sing}(Y)| \rightarrow Y$ is a weak homotopy equivalence.
\end{theorem}

\subsection{Left Bousfield localization}\label{section2.3} We give an overview of left Bousfield localization from \cite[Chapter 4]{Hirschhorn2003}. Let $\mathcal{M}$ be a simplicial model category; in particular, for any two objects $X$ and $Y$ of $\mathcal{M}$, we have a function complex $\operatorname{Map}_{\mathcal{M}}(X, Y)$. Let $T$ be a set of morphisms in $\mathcal{M}$.

\begin{definition}\label{deflblsdef}

\begin{enumerate}

\item An object $K$ of $\mathcal{M}$ is \emph{$T$-local} if $K$ is fibrant in $\mathcal{M}$ and, for every morphism $\lambda \colon A \rightarrow B$ in $T$, the map $\operatorname{Map}_{\mathcal{M}}(\lambda, K) \colon \operatorname{Map}_{\mathcal{M}}(B, K) \rightarrow \operatorname{Map}_{\mathcal{M}}(A, K)$ is a weak equivalence.

\item A morphism $f \colon X \rightarrow Y$ in $\mathcal{M}$ is a \emph{$T$-local equivalence} if, for every $T$-local object $K$ of $\mathcal{M}$, the map $\operatorname{Map}_{\mathcal{M}}(f, K) \colon \operatorname{Map}_{\mathcal{M}}(Y, K) \rightarrow \operatorname{Map}_{\mathcal{M}}(X, K)$ is a weak equivalence.

\end{enumerate}

\end{definition}

\begin{remark}\label{deflblsdef1} Every morphism in $T$ and every weak equivalence in $\mathcal{M}$ is a $T$-local equivalence.
\end{remark}
The left Bousfield localization of $\mathcal{M}$ at $T$ adds the morphisms in $T$ to the class of weak equivalences in $\mathcal{M}$.

\begin{theorem}[{\cite[Theorem 4.7]{Barwick2010}}]\label{lblexists} Let $\mathcal{M}$ be a combinatorial, left proper, and simplicial model category. Let $T$ be a set of morphisms in $\mathcal{M}$. There is a model structure $\mathcal{L}_{T}\mathcal{M}$ on the underlying category of $\mathcal{M}$, called the \emph{left Bousfield localization of $\mathcal{M}$ at $T$}, in which a morphism $f$ is

\begin{enumerate}

\item a weak equivalence if and only if $f$ is a $T$-local equivalence; and

\item a cofibration if and only if $f$ is a cofibration in $\mathcal{M}$.

\end{enumerate}
An object $K$ is fibrant in $\mathcal{L}_{T}\mathcal{M}$ if and only if $K$ is $T$-local. The model category $\mathcal{L}_{T}\mathcal{M}$ is combinatorial and left proper, and it inherits the simplicial structure of $\mathcal{M}$.
\end{theorem}

\subsection{Cartesian model categories}\label{section2.4} Lastly, we go over cartesian model categories following \cite[Chapter 2.9]{Bergner2018}.

\begin{definition}\label{defcccdef} A category $\mathcal{C}$ is \emph{cartesian closed} if it has finite products and, for any two objects $X$ and $Y$, an internal function object $Y^{X}$ satisfying, for any object $Z$, a natural isomorphism \[\operatorname{Hom}_{\mathcal{C}}(Z, Y^{X}) \cong \operatorname{Hom}_{\mathcal{C}}(Z \times X, Y).\]
\end{definition}
A cartesian model category is a cartesian closed category with a compatible model structure.

\begin{definition}[{\cite[\S 2.2]{Rezk2010}}]\label{defcmcdef} A model category $\mathcal{M}$ is \emph{cartesian} if its underlying category is cartesian closed, its terminal object is cofibrant, and, given a cofibration $i \colon A \rightarrow A'$ and a fibration $p \colon X' \rightarrow X$, the induced map \[q \colon (X')^{A'} \rightarrow (X')^{A} \underset{X^{A}}{\times} X^{A'}\] is a fibration that is acyclic if $i$ or $p$ is acyclic.
\end{definition}
The model category ${\mathcal{SS}\mathrm{ets}}$ is cartesian, for its cartesian structure coincides with its simplicial structure.

\section{Rational homotopy theory of non-simply connected spaces}\label{section3}
From now on, let $M$ be a multiplicative subset of $\mathbb{Z}$. We study the rational homotopy theory of non-simply connected spaces. In particular, we focus on the spaces whose higher homotopy groups are $M^{-1}\mathbb{Z}$-modules; we refer to such spaces as \emph{$M$-local spaces}. When $M^{-1}\mathbb{Z}=\mathbb{Q}$, we arrive at the study of rational spaces. Our main result is a characterization of $M$-local spaces in terms of function complexes, which we use to left Bousfield localize the model category of simplicial sets and get a model category whose fibrant objects are the $M$-local Kan complexes. In particular, we obtain a model category whose fibrant objects are the rational Kan complexes.

\subsection{The construction of local spheres}\label{section3.1}
Let $M$ be a multiplicative subset of $\mathbb{Z}$ and let $n \geq 2$. Our work in this section makes reference to the CW complex construction of the {$M$-local $n$-sphere} $S^{n}_{M^{-1}\mathbb{Z}}$ from \cite[Chapter 9(a)]{FHT2001}, which subsumes the construction of the {rational $n$-sphere} $S^{n}_{\mathbb{Q}}$. We go over this construction.

Let $\mathbb{N}$ denote the set of positive integers. We enumerate the set $M \cap \mathbb{N}=\{m_{1}, m_{2}, m_{3}, \dots\}$. We denote the $n$-sphere by $S^{n}$ and the $(n+1)$-disk by $D^{n+1}$. The \emph{$M$-local $n$-sphere} $S^{n}_{M^{-1}\mathbb{Z}}$ is constructed as the pushout
\[\begin{tikzcd}[row sep=3mm, column sep=3mm, ampersand replacement=\&]
	{\coprod\limits_{j=1}^{\infty} (S^{n})_{j}} \&\& {\bigvee\limits_{i=0}^{\infty} (S^{n})_{i}} \\
	\\
	{\coprod\limits_{j=1}^{\infty} (D^{n+1})_{j}} \&\& {S^{n}_{M^{-1}\mathbb{Z}},}
	\arrow["h", from=1-1, to=1-3]
	\arrow[from=1-1, to=3-1]
	\arrow[from=1-3, to=3-3]
	\arrow[from=3-1, to=3-3]
\end{tikzcd}\]
where, for $j \geq 1$, the $j$-th $(n+1)$-cell $(D^{n+1})_{j}$ is attached to $\bigvee\limits_{i=0}^{\infty} (S^{n})_{i}$ by an attaching map $h$ given by
\[S^{n} \xrightarrow{\textrm{pinch}} S^{n} \vee S^{n} \xrightarrow{\operatorname{id}_{S^{n}} \vee (-m_{j})} (S^{n})_{j-1} \vee (S^{n})_{j} \xrightarrow{\textrm{inclusion}} \bigvee\limits_{i=0}^{\infty} (S^{n})_{i}.\]
The map $(-m_{j})$ is a self-map of $S^{n}$ that has degree $(-m_{j})$. In other words, the simply connected CW complex \[S^{n}_{M^{-1}\mathbb{Z}}=\left(\bigvee\limits_{i=0}^{\infty} (S^{n})_{i}\right) \underset{h}{\bigcup} \left(\coprod\limits_{j=1}^{\infty} (D^{n+1})_{j}\right)\] is the mapping telescope of the sequence of maps $S^{n} \xrightarrow{m_{1}} S^{n} \xrightarrow{m_{2}} S^{n} \xrightarrow{m_{3}} \cdots$.

\subsection{A characterization of rational spaces}\label{section3.2}
Let $M$ be a multiplicative subset of $\mathbb{Z}$. A simply connected space is \emph{$M$-local} if its higher homotopy groups are $M^{-1}\mathbb{Z}$-modules. This notion can be extended to all spaces.

\begin{definition}[{\cite[Definition 6.1]{GTHT2000}}]\label{deflocalnonscdef} Let $M$ be a multiplicative subset of $\mathbb{Z}$. A space $Y$ is \emph{$M$-local} if, for every $n \geq 2$, the abelian group $\pi_{n}(Y)$ is an $M^{-1}\mathbb{Z}$-module. In the case when $M^{-1}\mathbb{Z}=\mathbb{Q}$, a space $Y$ is said to be \emph{rational} if, for every $n \geq 2$, the abelian group $\pi_{n}(Y)$ is a rational vector space.
\end{definition}

\begin{example}\label{deflocalnonscdef1} The spaces $S^{1}$, $S^{1} \vee S^{1}$, and $\mathbb{R}P^{\infty}$ are rational because their higher homotopy groups vanish.
\end{example}
We use the construction of $S^{n}_{M^{-1}\mathbb{Z}}$ from Section \ref{section3.1}, including our notation, to prove the following key result.

\begin{proposition}\label{thebigdealnonsc} Let $M$ be a multiplicative subset of $\mathbb{Z}$ and let $n \geq 2$. Let $Y$ be a space such that $\pi_{n}(Y)$ is an $M^{-1}\mathbb{Z}$-module. Then, every map $f \colon S^{n} \rightarrow Y$ extends uniquely up to homotopy to a map $g \colon S^{n}_{M^{-1}\mathbb{Z}} \rightarrow Y$ such that the diagram

\[\begin{tikzcd}[row sep=3mm, column sep=3mm, ampersand replacement=\&]
	{S^{n}} \&\& {S^{n}_{M^{-1}\mathbb{Z}}} \\
	\\
	\&\& {Y}
	\arrow[from=1-1, to=1-3]
	\arrow["f"', from=1-1, to=3-3]
	\arrow["g"', dashed, from=1-3, to=3-3]
\end{tikzcd}\]
commutes.
\end{proposition}

\begin{proof} Given a map $f \colon S^{n} \rightarrow Y$, we know that $[f]$ is an element of the $M^{-1}\mathbb{Z}$-module $\pi_{n}(Y)$. Thus, there exists a unique element $[u]$ of $\pi_{n}(Y)$ such that $[f]=m_{1} \cdot [u]$, and the representative $u \colon S^{n} \rightarrow Y$ of $[u]$ is unique up to homotopy. Using the notation in our construction of $S^{n}_{M^{-1}\mathbb{Z}}$, we extend $f$ from $S^{n}=(S^{n})_{0}$ to $(S^{n})_{0} \vee (S^{n})_{1}$ uniquely up to homotopy by setting $g|_{(S^{n})_{1}}=u$. Since $[f]-(m_{1} \cdot [u])=0$ in $\pi_{n}(Y)$, we can further extend $f$ to \[((S^{n})_{0} \vee (S^{n})_{1}) \underset{h}{\bigcup} (D^{n+1})_{1}\] uniquely up to homotopy. Continuing further down the telescope in this fashion yields the claim.
\end{proof}
For two topological spaces $X$ and $Y$, the {function complex} $\operatorname{Map}_{\mathcal{T}\mathrm{op}}(X, Y)$ is the simplicial set with $n$-simplices \[\operatorname{Map}_{\mathcal{T}\mathrm{op}}(X, Y)_{n}=\operatorname{Hom}_{\mathcal{T}\mathrm{op}}(X \times |\Delta[n]|, Y).\]
Note that $\pi_{0}(\operatorname{Map}_{\mathcal{T}\mathrm{op}}(X, Y))=[X, Y]$ is the set of homotopy classes of maps from $X$ to $Y$. We provide our promised characterization of $M$-local spaces, and thus that of rational spaces, in terms of function complexes.

\begin{proposition}\label{illdothisonenonsc} Let $M$ be a multiplicative subset of $\mathbb{Z}$. A space $Y$ is $M$-local if and only if, for every $n \geq 2$, the restriction map $\operatorname{Map}_{\mathcal{T}\mathrm{op}}(S^{n}_{M^{-1}\mathbb{Z}}, Y) \rightarrow \operatorname{Map}_{\mathcal{T}\mathrm{op}}(S^{n}, Y)$ is a weak equivalence.
\end{proposition}

\begin{proof} If $Y$ is $M$-local, then Proposition \ref{thebigdealnonsc} implies that $\operatorname{Map}_{\mathcal{T}\mathrm{op}}(S^{n}_{M^{-1}\mathbb{Z}}, Y) \rightarrow \operatorname{Map}_{\mathcal{T}\mathrm{op}}(S^{n}, Y)$ is a weak equivalence for every $n \geq 2$. Conversely, for $n \geq 2$, let $[f]$ be an element of $\pi_{n}(Y)$ represented by a map $f \colon S^{n} \rightarrow Y$, and let $m \in M$. The induced bijection of path components $[S^{n}_{M^{-1}\mathbb{Z}}, Y] \xrightarrow{\cong} [S^{n}, Y]$ informs us that there exists a map $g \colon S^{n}_{M^{-1}\mathbb{Z}} \rightarrow Y$, unique up to homotopy, such that the diagram
\[\begin{tikzcd}[row sep=3mm, column sep=3mm, ampersand replacement=\&]
	{S^{n}} \&\& {S^{n}_{M^{-1}\mathbb{Z}}} \\
	\\
	\&\& Y
	\arrow["{\tau}", from=1-1, to=1-3]
	\arrow["f"', from=1-1, to=3-3]
	\arrow["g"', from=1-3, to=3-3]
\end{tikzcd}\]
commutes up to homotopy, where $\tau$ denotes the inclusion. Homotopy commutativity implies that $[f]=g_{\ast}([\tau])$. As $S^{n}_{M^{-1}\mathbb{Z}}$ is $M$-local, there is a unique element $\dfrac{1}{m} \cdot [\tau]$ in $\pi_{n}(S^{n}_{M^{-1}\mathbb{Z}})$, and we define
$\dfrac{1}{m} \cdot [f]=g_{\ast}\left(\dfrac{1}{m} \cdot [\tau]\right)$.
\end{proof}

\subsection{The model category for rational spaces}\label{section3.3}
Let $M$ be a multiplicative subset of $\mathbb{Z}$. Our characterization of $M$-local spaces in Proposition \ref{illdothisonenonsc} prompts us to left Bousfield localize ${\mathcal{SS}\mathrm{ets}}$ at the set \[T_{M^{-1}\mathbb{Z}}=\{\operatorname{Sing}(S^{n}) \rightarrow \operatorname{Sing}(S^{n}_{M^{-1}\mathbb{Z}}) \mid n \geq 2 \}.\] We characterize the $T_{M^{-1}\mathbb{Z}}$-local objects in terms of the counterpart of Definition \ref{deflocalnonscdef} for simplicial sets.

\begin{definition}\label{defrhtssetsnonscdef} Let $M$ be a multiplicative subset of $\mathbb{Z}$. A simplicial set $K$ is \emph{$M$-local} if $|K|$ is an $M$-local space. A simplicial set $K$ is \emph{rational} if $|K|$ is a rational space.
\end{definition}
To transfer our work in Section \ref{section3.2} from topological spaces to simplicial sets, we use the following proposition.

\begin{proposition}[{\cite[Proposition 1.1.11]{Hirschhorn2003}}]\label{pshhelpmenonsc} If $A$ and $X$ are simplicial sets and $X$ is a Kan complex, then there is a natural weak equivalence $\operatorname{Map}_{\mathcal{SS}\mathrm{ets}}(A, X) \rightarrow \operatorname{Map}_{\mathcal{T}\mathrm{op}}(|A|, |X|)$.
\end{proposition}
The following lemma also underpins our arguments in this section.

\begin{lemma}[{\cite[Corollary 9.3.3(2)]{Hirschhorn2003}}]\label{citeadnauseam} If $f \colon X \rightarrow Y$ is a weak homotopy equivalence of CW complexes and $Z$ is a topological space, then $\operatorname{Map}_{\mathcal{T}\mathrm{op}}(f, Z) \colon \operatorname{Map}_{\mathcal{T}\mathrm{op}}(Y, Z) \rightarrow \operatorname{Map}_{\mathcal{T}\mathrm{op}}(X, Z)$ is a weak equivalence.
\end{lemma}
We show that a simplicial set is $T_{M^{-1}\mathbb{Z}}$-local if and only if it is an $M$-local Kan complex.

\begin{proposition}\label{charthelocalsnonsc} Let $M$ be a multiplicative subset of $\mathbb{Z}$. Let $K$ be a Kan complex. The following statements are equivalent.

\begin{enumerate}

\item The map $\operatorname{Map}_{\mathcal{SS}\mathrm{ets}}(\operatorname{Sing}(S^{n}_{M^{-1}\mathbb{Z}}), K) \rightarrow \operatorname{Map}_{\mathcal{SS}\mathrm{ets}}(\operatorname{Sing}(S^{n}), K)$ is a weak equivalence for all $n \geq 2$.

\item The map $\operatorname{Map}_{\mathcal{T}\mathrm{op}}(|\operatorname{Sing}(S^{n}_{M^{-1}\mathbb{Z}})|, |K|) \rightarrow \operatorname{Map}_{\mathcal{T}\mathrm{op}}(|\operatorname{Sing}(S^{n})|, |K|)$ is a weak equivalence for all $n \geq 2$.

\item The map $\operatorname{Map}_{\mathcal{T}\mathrm{op}}(S^{n}_{M^{-1}\mathbb{Z}}, |K|) \rightarrow \operatorname{Map}_{\mathcal{T}\mathrm{op}}(S^{n}, |K|)$ is a weak equivalence for all $n \geq 2$.

\item The space $|K|$ is $M$-local.

\end{enumerate}

\end{proposition}

\begin{proof} Firstly, (1) and (2) are equivalent by Proposition \ref{pshhelpmenonsc}. Then, (2) and (3) are equivalent by Theorem \ref{quilleneqsecretly} and Lemma \ref{citeadnauseam}. Lastly, (3) and (4) are equivalent by Proposition \ref{illdothisonenonsc}.
\end{proof}
Proposition \ref{charthelocalsnonsc} describes the left Bousfield localization of ${\mathcal{SS}\mathrm{ets}}$ at the set $T_{M^{-1}\mathbb{Z}}$.

\begin{theorem}\label{ididitididitnonsc} Let $M$ be a multiplicative subset of $\mathbb{Z}$. There is a combinatorial, left proper, and simplicial model structure $\mathcal{L}_{M^{-1}\mathbb{Z}}\mathcal{SS}\mathrm{ets}$ on the category of simplicial sets in which a map $f$ is

\begin{enumerate}

\item a weak equivalence if and only if, for every $M$-local Kan complex $K$, the map $\operatorname{Map}_{\mathcal{SS}\mathrm{ets}}(f, K)$ is a weak equivalence; and

\item a cofibration if and only if $f$ is a monomorphism.

\end{enumerate}
A simplicial set $K$ is fibrant in $\mathcal{L}_{M^{-1}\mathbb{Z}}\mathcal{SS}\mathrm{ets}$ if and only if $K$ is an $M$-local Kan complex.
\end{theorem}
In particular, we get a model category whose fibrant objects are the rational Kan complexes.

\section{Simplicial spaces and complete Segal spaces}\label{section4}
To apply our work in Section \ref{section3} to the study of $(\infty, 1)$-categories, we review the first model for $(\infty, 1)$-categories that we use in this paper, which is that of complete Segal spaces. Complete Segal spaces are simplicial spaces satisfying a condition that captures morphism composition up to homotopy. After giving some background on simplicial spaces and their Reedy model structure, we recall Rezk's cartesian model category for complete Segal spaces, which is constructed as a left Bousfield localization of the Reedy model structure on simplicial spaces.

\subsection{Simplicial spaces}\label{section4.1} We collect some background on simplicial spaces and their Reedy model structure.

\begin{definition}\label{defsimpobjmodeldef} A \emph{simplicial space} is a functor $X \colon \Delta^{\mathrm{op}} \rightarrow \mathcal{SS}\mathrm{ets}$.
\end{definition}
We write $\mathcal{SS}\mathrm{ets}^{\Delta^{\mathrm{op}}}$ for the category of simplicial spaces and their simplicial maps (natural transformations).

\begin{example}\label{defsimpobjmodeldef1} We can view a simplicial space $X \colon \Delta^{\mathrm{op}} \rightarrow \mathcal{SS}\mathrm{ets}$ as a bisimplicial set $X \colon \Delta^{\mathrm{op}} \times \Delta^{\mathrm{op}} \rightarrow \mathcal{S}\mathrm{ets}$. Thus, given a simplicial set $Z \colon \Delta^{\mathrm{op}} \rightarrow \mathcal{S}\mathrm{ets}$, we can view it as a simplicial space in two ways.

\begin{enumerate}

\item The \emph{transpose} simplicial space $Z^{t}$ is the composite of functors $\Delta^{\mathrm{op}} \times \Delta^{\mathrm{op}} \xrightarrow{\operatorname{pr}_{1}} \Delta^{\mathrm{op}} \xrightarrow{Z} \mathcal{S}\mathrm{ets}$, where $\operatorname{pr}_{1}$ denotes the projection to the first coordinate. Note that $(Z^{t})_{n}=Z_{n}$ is discrete for every $n \geq 0$.

\item The \emph{constant} simplicial space $Z$ is the composite of functors $\Delta^{\mathrm{op}} \times \Delta^{\mathrm{op}} \xrightarrow{\operatorname{pr}_{2}} \Delta^{\mathrm{op}} \xrightarrow{Z} \mathcal{S}\mathrm{ets}$, where $\operatorname{pr}_{2}$ denotes the projection to the second coordinate.
\end{enumerate}

\end{example}
In the Reedy model structure on $\mathcal{SS}\mathrm{ets}^{\Delta^{\mathrm{op}}}$, the weak equivalences are the levelwise weak equivalences, the cofibrations are the monomorphisms, and the fibrations are described using coskeleta; see \cite[Theorem A]{Reedy1974} and {\cite[Theorems 2.6.6 and 2.6.10]{Bergner2018}}. This model structure is combinatorial, left proper, and simplicial; in particular, given two simplicial spaces $X$ and $Y$, we define their function complex $\operatorname{Map}_{\mathcal{SS}\mathrm{ets}^{\Delta^{\mathrm{op}}}}(X, Y)$ by \[\operatorname{Map}_{\mathcal{SS}\mathrm{ets}^{\Delta^{\mathrm{op}}}}(X, Y)_{n}=\operatorname{Hom}_{\mathcal{SS}\mathrm{ets}^{\Delta^{\mathrm{op}}}}(X \times \Delta[n], Y).\]
Then, applying the Yoneda lemma to Example \ref{defsimpobjmodeldef1} yields the following useful result for simplicial spaces.

\begin{proposition}\label{yonedahelpsus} Let $X$ be a simplicial space. Let $n$ be a non-negative integer. There is a natural isomorphism \[X_{n} \cong \operatorname{Map}_{\mathcal{SS}\mathrm{ets}^{\Delta^{\mathrm{op}}}}(\Delta[n]^{t}, X).\]
\end{proposition}
In particular, we note the following consequence of Proposition \ref{yonedahelpsus} for later use.

\begin{corollary}\label{usein8.2} If $X$ is a simplicial space and $Z$ is a simplicial set, then there is a natural isomorphism \[\operatorname{Map}_{\mathcal{SS}\mathrm{ets}^{\Delta^{\mathrm{op}}}}(\Delta[n]^{t} \times Z, X) \cong \operatorname{Map}_{\mathcal{SS}\mathrm{ets}}(Z, X_{n}).\]
\end{corollary}
A property of the Reedy model structure on simplicial spaces that is relevant to our work is that it is cartesian \cite[Example 2.9.4]{Bergner2018}. In particular, given two simplicial spaces $X$ and $Y$, we have a simplicial space $Y^{X}$, and if $Y$ is Reedy fibrant, then so is $Y^{X}$. Although the localization of a cartesian model category need not be cartesian, we have a convenient criterion for localizations of the Reedy model structure on simplicial spaces.

\begin{proposition}[{\cite[Proposition 9.2]{Rezk2001}}]\label{rezkcriterionprop} Let $T$ be a set of simplicial maps. Suppose that, for every $T$-local simplicial space $K$, the simplicial space $K^{\Delta[1]^{t}}$ is $T$-local. Then, the left Bousfield localization $\mathcal{L}_{T}\mathcal{SS}\mathrm{ets}^{\Delta^{\mathrm{op}}}$ of the Reedy model structure on $\mathcal{SS}\mathrm{ets}^{\Delta^{\mathrm{op}}}$ at the set $T$ is cartesian.
\end{proposition}

\subsection{Complete Segal spaces}\label{section4.2} We give an overview of complete Segal spaces from \cite[Chapter 5]{Bergner2018} and \cite{Rezk2001}.

Let $n \geq 2$ and $0 \leq i \leq n-1$. We define the order-preserving map $\alpha^{i} \colon [1] \rightarrow [n]$ by $\alpha^{i}(x)=i+x$. Let \[G(n)=\bigcup_{i=0}^{n-1} \alpha^{i}(\Delta[1]) \subseteq \Delta[n].\] Then, for a simplicial space $W$, Proposition \ref{yonedahelpsus} implies that the inclusion $G(n) \rightarrow \Delta[n]$ induces a map \[\varphi_{n} \colon W_{n} \rightarrow \underbrace{W_{1} \underset{W_{0}}{\times} W_{1} \underset{W_{0}}{\times} \cdots \underset{W_{0}}{\times} W_{1}}_{n},\]
which is called the \emph{$n$-th Segal map} of $W$.

\begin{definition}[{\cite[\S 4.1]{Rezk2001}}]\label{segalobjectdef} A simplicial space $W$ is a \emph{Segal space} if $W$ is Reedy fibrant and, for every $n \geq 2$, its $n$-th Segal map is a weak equivalence.
\end{definition}
The categorical interpretation of Segal spaces is distilled in the following definition.

\begin{definition}[{\cite[\S 5.1]{Rezk2001}}]\label{defobjmapdef} Let $W$ be a Segal space.

\begin{enumerate}

\item The \emph{set of objects} of $W$ is $\operatorname{ob}(W)=W_{0,0}$.

\item If $x$ and $y$ are objects of $W$, then their \emph{mapping space} $\operatorname{map}_{W}(x, y)$ is the pullback

\[\begin{tikzcd}[row sep=3mm, column sep=3mm, ampersand replacement=\&]
	{\operatorname{map}_{W}(x, y)} \&\& {W_{1}} \\
	\\
	{\Delta[0]} \&\& {W_{0} \times W_{0}.}
	\arrow[from=1-1, to=1-3]
	\arrow[from=1-1, to=3-1]
	\arrow["{(d_{1}, d_{0})}"', from=1-3, to=3-3]
	\arrow["{(x, y)}", from=3-1, to=3-3]
\end{tikzcd}\]

\end{enumerate}

\end{definition}
As the Segal maps are weak equivalences, we can define a composition operation on mapping spaces that is well-defined, associative, and unital up to a notion of homotopy. This notion of homotopy gives rise to the homotopy category of a Segal space. A Dwyer-Kan equivalence is a map that preserves this structure; the name refers to Dwyer and Kan's analogous notion of equivalence for simplicial categories from \cite[\S 2.4]{DK1980}.

\begin{definition}[{\cite[\S 7.4]{Rezk2001}}]\label{defdkdef} A map of Segal spaces $f \colon X \rightarrow Y$ is a \emph{Dwyer-Kan equivalence} if,

\begin{enumerate}

\item for all $x, y \in \operatorname{ob}(X)$, the map $\operatorname{map}_{X}(x, y) \rightarrow \operatorname{map}_{Y}(f(x), f(y))$ is a weak equivalence; and

\item the induced functor on homotopy categories $\operatorname{Ho}(X) \rightarrow \operatorname{Ho}(Y)$ is essentially surjective.

\end{enumerate}

\end{definition}
To define complete Segal spaces, let $I$ be the category with two objects and a unique isomorphism between them, and let $E=\operatorname{nerve}(I)$. For a Segal space $W$, the map $E \rightarrow \Delta[0]$ induces a map $W_{0} \rightarrow \operatorname{Map}_{\mathcal{SS}\mathrm{ets}^{\Delta^{\mathrm{op}}}}(E^{t}, W)$.

\begin{definition}[{\cite[\S 6]{Rezk2001}}]\label{charcompleteness} A Segal space $W$ is \emph{complete} if $W_{0} \rightarrow \operatorname{Map}_{\mathcal{SS}\mathrm{ets}^{\Delta^{\mathrm{op}}}}(E^{t}, W)$ is a weak equivalence.
\end{definition}
Localizing the Reedy model structure on simplicial spaces at the set $\{G(n)^{t} \rightarrow \Delta[n]^{t} \mid n \geq 2\} \cup \{E^{t} \rightarrow \Delta[0]^{t}\}$ yields a model category whose fibrant objects are the complete Segal spaces.

\begin{theorem}[{\cite[Theorem 7.2]{Rezk2001}}]\label{compsegalspacemodellbl} There is a combinatorial, left proper, and simplicial model structure $\mathcal{CSS}$ on the category of simplicial spaces in which a map $f$ is

\begin{enumerate}

\item a weak equivalence if and only if, for every complete Segal space $W$, the map $\operatorname{Map}_{\mathcal{SS}\mathrm{ets}^{\Delta^{\mathrm{op}}}}(f, W)$ is a weak equivalence; and

\item a cofibration if and only if $f$ is a monomorphism.

\end{enumerate}
A simplicial space $W$ is fibrant in $\mathcal{CSS}$ if and only if $W$ is a complete Segal space.
\end{theorem}
By an application of Proposition \ref{rezkcriterionprop}, Rezk also extracts the compatibility of $\mathcal{CSS}$ with the cartesian closure.

\begin{theorem}[{\cite[Theorem 7.2]{Rezk2001}}]\label{cssegalspacecartesian} The model category $\mathcal{CSS}$ is cartesian.
\end{theorem}
Moreover, we can complete a Segal space up to Dwyer-Kan equivalence in a functorial fashion as follows.

\begin{theorem}[{\cite[\S 14]{Rezk2001}}]\label{completionfunctor} Let $W$ be a Segal space. Then, there exist a complete Segal space $\widehat{W}$ and a Dwyer-Kan equivalence $W \rightarrow \widehat{W}$. Moreover, this assignment is functorial.
\end{theorem}
This result helps establish a categorical interpretation of the weak equivalences in $\mathcal{CSS}$ between Segal spaces.

\begin{theorem}[{\cite[Theorem 7.7]{Rezk2001}}]\label{tractablecat} A map of Segal spaces is a weak equivalence in $\mathcal{CSS}$ if and only if it is a Dwyer-Kan equivalence.
\end{theorem}

\section{Segal categories}\label{section5}
We review another model for $(\infty, 1)$-categories, that of Segal categories. Segal categories are simplicial spaces; however, in place of the completeness condition, we ask that the space of $0$-simplices of a Segal category be discrete. We recall Bergner's model categories for Segal categories, and we show that Bergner's model category for Reedy fibrant Segal categories can be left-induced from Rezk's model category for complete Segal spaces. We also go over Bergner's proof of the equivalence between complete Segal spaces and Segal categories.

\subsection{Definition of a Segal category}\label{section5.1}
We first gather the definitions that we need from \cite[Chapter 6.1]{Bergner2018}.

\begin{definition}\label{defsegalprecatdef} A \emph{Segal precategory} is a simplicial space $X$  whose space of $0$-simplices $X_{0}$ is discrete.
\end{definition}
We write $\mathcal{S}\mathrm{e}\mathcal{C}\mathrm{at}$ for the category of Segal precategories and their simplicial maps (natural transformations).

\begin{definition}\label{defsegalcatdef} A \emph{Segal category} is a Segal precategory $X$ such that, for every $n \geq 2$, its $n$-th Segal map is a weak equivalence.
\end{definition}
A Segal category is not required to be Reedy fibrant, so it need not be a Segal space. Still, objects and mapping spaces are defined as in Segal spaces. The discreteness of the space of $0$-simplices has a key consequence.

\begin{proposition}\label{disjointoffibers} If $X$ is a Segal category, then its space of $1$-simplices decomposes as $X_{1}=\coprod\limits_{(x, y)} \operatorname{map}_{X}(x, y)$. Thus, for every $n \geq 2$, the $n$-th Segal map of $X$ is the weak equivalence \[\varphi_{n} \colon X_{n} \rightarrow \coprod\limits_{(x_{0}, \dots, x_{n})} \prod\limits_{i=0}^{n-1} \operatorname{map}_{X}(x_{i}, x_{i+1}).\]
\end{proposition}

\subsection{Model categories for Segal categories}\label{section5.2}
Unlike the Reedy model structure on simplicial spaces, there exists no model structure on Segal precategories with levelwise weak equivalences and levelwise cofibrations. Thus, the model categories for Segal categories that we rely on do not arise as left Bousfield localizations. In \cite{Bergner2007}, Bergner directly develops two model categories for Segal categories. The weak equivalences are described by means of the following result, which refers to the model category for complete Segal spaces $\mathcal{CSS}$.

\begin{proposition}[{\cite[\S 5]{Bergner2007}}]\label{localizationofsepre} There is a functor $L_{c} \colon \mathcal{S}\mathrm{e}\mathcal{C}\mathrm{at} \rightarrow \mathcal{S}\mathrm{e}\mathcal{C}\mathrm{at}$ such that, for every Segal precategory $X$, its image $L_{c}(X)$ is a Segal category and a Segal space that is weakly equivalent to $X$ in $\mathcal{CSS}$.
\end{proposition}

\begin{definition}[{\cite[\S 5]{Bergner2007}}]\label{defdklcdef} A map of Segal precategories $f$ is a \emph{Dwyer-Kan equivalence} if $L_{c}(f)$ is a Dwyer-Kan equivalence of Segal spaces.
\end{definition}
In this terminology, Proposition \ref{localizationofsepre} has the following consequence.

\begin{lemma}\label{iffforlc} Let $\psi$ be a map of Segal precategories. Then, $\psi$ is a weak equivalence in $\mathcal{CSS}$ if and only if $\psi$ is a Dwyer-Kan equivalence.
\end{lemma}
We also record for later use a technical result that features in Bergner's arguments.

\begin{lemma}[{\cite[Lemma 5.9]{Bergner2007}}]\label{technical} If a map of Segal precategories $\psi$ has the right lifting property against the class of monomorphisms of Segal precategories, then $\psi$ is a Dwyer-Kan equivalence.
\end{lemma}
We describe the model category for Segal categories that we mainly use; the simplicial structure is established in \cite[Proposition 6.3]{Bergner2017}, and the fibrant objects are described in \cite[Corollary 5.13]{Bergner2007} and \cite[Theorem 3.2]{Bergner2007_2}.

\begin{theorem}[{{\cite[Theorem 5.1]{Bergner2007}}; {\cite[Theorem 6.4.4]{Pellissier2003}}}]\label{secatcmodel} There is a combinatorial, left proper, and simplicial model structure $\mathcal{S}\mathrm{e}\mathcal{C}\mathrm{at}_{c}$ on the category of Segal precategories in which a map $f$ is

\begin{enumerate}

\item a weak equivalence if and only if $f$ is a Dwyer-Kan equivalence; and

\item a cofibration if and only if $f$ is a monomorphism.

\end{enumerate}
A Segal precategory $X$ is fibrant in $\mathcal{S}\mathrm{e}\mathcal{C}\mathrm{at}_{c}$ if and only if $X$ is a Reedy fibrant Segal category.
\end{theorem}
Bergner also develops a Quillen equivalent model category for levelwise fibrant Segal categories; its left properness is proven in \cite[Proposition 6.5.6]{Bergner2018}, and its fibrant objects are described in \cite[Theorem 4.2]{Bergner2007_2}.

\begin{theorem}[{{\cite[Theorems 7.1 and 7.5]{Bergner2007}}}]\label{secatfmodel} There is a combinatorial and left proper model structure $\mathcal{S}\mathrm{e}\mathcal{C}\mathrm{at}_{f}$ on the category of Segal precategories in which

\begin{enumerate}

\item a map $f$ is a weak equivalence if and only if $f$ is a Dwyer-Kan equivalence; and

\item the fibrant objects are the levelwise fibrant Segal categories.

\end{enumerate}
Moreover, the adjunction $\operatorname{id} \colon \mathcal{S}\mathrm{e}\mathcal{C}\mathrm{at}_{f} \rightleftarrows \mathcal{S}\mathrm{e}\mathcal{C}\mathrm{at}_{c} \colon \operatorname{id}$ is a Quillen equivalence.
\end{theorem}

\subsection{The equivalence with complete Segal spaces}\label{section5.3}
In \cite{Bergner2007}, Bergner gives a Quillen equivalence between complete Segal spaces and Segal categories by means of a right adjoint $R$ to the inclusion $I \colon \mathcal{S}\mathrm{e}\mathcal{C}\mathrm{at}_{c} \rightarrow \mathcal{CSS}$. This right adjoint sends a simplicial space $W$ to the Segal precategory $R(W)$ defined as the pullback
\[\begin{tikzcd}[row sep=3mm, column sep=3mm, ampersand replacement=\&]
	{R(W)} \&\& {\operatorname{cosk}_{0}(W_{0, 0})} \\
	\\
	W \&\& {\operatorname{cosk}_{0}(W_{0}).}
	\arrow[from=1-1, to=1-3]
	\arrow[from=1-1, to=3-1]
	\arrow[from=1-3, to=3-3]
	\arrow[from=3-1, to=3-3]
\end{tikzcd}\]

\begin{theorem}[{\cite[Theorem 6.3]{Bergner2007}}]\label{quilleneqinfty1} The adjunction $I \colon \mathcal{S}\mathrm{e}\mathcal{C}\mathrm{at}_{c} \rightleftarrows \mathcal{CSS} \colon R$ is a Quillen equivalence.
\end{theorem}
In particular, we can replace a complete Segal space by a Segal category up to Dwyer-Kan equivalence.

\begin{theorem}[{\cite[Theorem 6.3]{Bergner2007}}]\label{quilleneqinfty1cor} If $W$ is a complete Segal space, then $R(W)$ is a Reedy fibrant Segal category and the map $R(W) \rightarrow W$ is a Dwyer-Kan equivalence.
\end{theorem}

\subsection{An approach by way of left transfer}\label{section5.4}
We show that the adjunction from Theorem \ref{quilleneqinfty1} allows us to left-induce the model category $\mathcal{S}\mathrm{e}\mathcal{C}\mathrm{at}_{c}$ from the model category $\mathcal{CSS}$ in the sense of the following theorem.

\begin{theorem}[{\cite[Theorem 2.23]{BHKKRS2015}}]\label{lefttransfertool} Let $F \colon \mathcal{D} \rightleftarrows \mathcal{M} \colon G$ be an adjunction, where $\mathcal{D}$ is a locally presentable category and $\mathcal{M}$ is a combinatorial model category with class of cofibrations $C$. Suppose that, if a morphism $\psi$ in $\mathcal{D}$ has the right lifting property against the class of maps $F^{-1}(C)$, then $F(\psi)$ is a weak equivalence. Then, there is a combinatorial model structure on $\mathcal{D}$, called the \emph{left-induced model structure}, in which a map $\psi$ is

\begin{enumerate}

\item a weak equivalence if and only if $F(\psi)$ is a weak equivalence in $\mathcal{M}$; and

\item a cofibration if and only if $F(\psi)$ is a cofibration in $\mathcal{M}$.

\end{enumerate}

\end{theorem}

\begin{theorem}\label{vialeftransfer} The model category $\mathcal{S}\mathrm{e}\mathcal{C}\mathrm{at}_{c}$ is left-induced from the model category $\mathcal{CSS}$ by means of the adjunction $I \colon \mathcal{S}\mathrm{e}\mathcal{C}\mathrm{at} \rightleftarrows \mathcal{CSS} \colon R$.
\end{theorem}

\begin{proof} Firstly, we verify that the condition in Theorem \ref{lefttransfertool} holds. If a map $\psi$ of Segal precategories has the right lifting property against the class of monomorphisms of Segal precategories, then $\psi$ is a Dwyer-Kan equivalence by Lemma \ref{technical}, and therefore $\psi$ is a weak equivalence in $\mathcal{CSS}$ by Lemma \ref{iffforlc}. The resulting left-induced model structure on the category $\mathcal{S}\mathrm{e}\mathcal{C}\mathrm{at}$ of Segal precategories has the same cofibrations as $\mathcal{S}\mathrm{e}\mathcal{C}\mathrm{at}_{c}$, as well as the same weak equivalences by Lemma \ref{iffforlc}. Hence, the left-induced model structure on $\mathcal{S}\mathrm{e}\mathcal{C}\mathrm{at}$ is $\mathcal{S}\mathrm{e}\mathcal{C}\mathrm{at}_{c}$.
\end{proof}

\section{Rational Segal categories}\label{section6}
Let $T$ be a set of simplicial maps. We introduce {$T$-local Segal categories}, which model $(\infty, 1)$-categories enriched in $T$-local spaces. In particular, we introduce rational Segal categories, which model rational $(\infty, 1)$-categories. We localize Bergner's model category for Reedy fibrant Segal categories to get a model category whose fibrant objects are the Reedy fibrant $T$-local Segal categories. Thus, we produce a model category whose fibrant objects are the Reedy fibrant rational Segal categories. The same localization yields a Quillen equivalent model category whose fibrant objects are the levelwise fibrant $T$-local Segal categories. In particular, we get a Quillen equivalent model category whose fibrant objects are the levelwise fibrant rational Segal categories.

\begin{definition}\label{defmlocalsecatdef} Let $T$ be a set of simplicial maps. A Segal category $X$ is \emph{$T$-local} if, for any two objects $x$ and $y$ of $X$, their mapping space $\operatorname{map}_{X}(x, y)$ is $T$-local. A Segal category $X$ is \emph{rational} if, for any two objects $x$ and $y$ of $X$, their mapping space $\operatorname{map}_{X}(x, y)$ is rational.
\end{definition}
Proposition \ref{disjointoffibers} implies that a Segal category $X$ is $T$-local if and only if $X$ is levelwise $T$-local, and we can encode the latter condition in the language of function complexes. To that end, we use the fact that the inclusion of Segal precategories into simplicial spaces has a left adjoint that takes in a simplicial space $X$ and collapses the space of $0$-simplices to its path components to output a Segal precategory $X_{r}$. Since Reedy fibrancy implies levelwise fibrancy \cite[Corollary 15.3.12(2)]{Hirschhorn2003}, we formulate our characterization under levelwise fibrancy.

\begin{proposition}\label{charlocalsecats} Let $T$ be a set of simplicial maps. Let $X$ be a levelwise fibrant Segal category. The following statements are equivalent.

\begin{enumerate}

\item For every $i \geq 0$ and every $\lambda$ in $T$, the map $\operatorname{Map}_{\mathcal{SS}\mathrm{ets}}(\lambda, X_{i})$ is a weak equivalence.

\item For every $i \geq 0$ and every $\lambda$ in $T$, the map $\operatorname{Map}_{\mathcal{SS}\mathrm{ets}^{\Delta^{\mathrm{op}}}}(\Delta[i]^{t} \times \lambda, X)$ is a weak equivalence.

\item For every $i \geq 0$ and every $\lambda$ in $T$, the map $\operatorname{Map}_{\mathcal{SS}\mathrm{ets}^{\Delta^{\mathrm{op}}}}((\Delta[i]^{t} \times \lambda)_{r}, X)$ is a weak equivalence.

\end{enumerate}

\end{proposition}

\begin{proof} Firstly, (1) and (2) are equivalent by Corollary \ref{usein8.2}. Then, (2) and (3) are equivalent by adjointness.
\end{proof}
Proposition \ref{charlocalsecats} instructs us to localize Bergner's model categories for Segal categories at the set \[\widetilde{T}=\{(\Delta[i]^{t} \times \lambda)_{r} \mid i \geq 0 \textrm{ and } \lambda \in T\}.\]

\begin{theorem}\label{ididitididitnonsconec} Let $T$ be a set of simplicial maps. There is a combinatorial, left proper, and simplicial model structure $\mathcal{L}_{T}\mathcal{S}\mathrm{e}\mathcal{C}\mathrm{at}_{c}$ on the category of Segal precategories in which a map $f$ is

\begin{enumerate}

\item a weak equivalence if and only if, for every Reedy fibrant $T$-local Segal category $X$, the induced map $\operatorname{Map}_{\mathcal{SS}\mathrm{ets}^{\Delta^{\mathrm{op}}}}(f, X)$ is a weak equivalence; and

\item a cofibration if and only if $f$ is a monomorphism.

\end{enumerate}
A Segal precategory $X$ is fibrant in $\mathcal{L}_{T}\mathcal{S}\mathrm{e}\mathcal{C}\mathrm{at}_{c}$ if and only if $X$ is a Reedy fibrant $T$-local Segal category.
\end{theorem}

\begin{corollary}\label{ididitididitqnonsconec} There is a model category whose fibrant objects are the Reedy fibrant rational Segal categories.
\end{corollary}

\begin{remark}\label{reducetormk} Localizing Bergner's model category for levelwise fibrant Segal categories at the same set $\widetilde{T}$ yields a Quillen equivalent model category whose fibrant objects are the levelwise fibrant $T$-local Segal categories. In particular, there is a model category whose fibrant objects are the levelwise fibrant rational Segal categories, and it is Quillen equivalent to the model category for Reedy fibrant rational Segal categories.
\end{remark}

\section{Rational complete Segal spaces}\label{section7}
Let $T$ be a set of simplicial maps. Now, we introduce {$T$-local complete Segal spaces}, which also model $(\infty, 1)$-categories enriched in $T$-local spaces. In particular, we introduce rational complete Segal spaces, which model rational $(\infty, 1)$-categories. We left Bousfield localize Rezk's model category for complete Segal spaces to obtain a model category whose fibrant objects are the $T$-local complete Segal spaces. Thus, we produce a model category whose fibrant objects are the rational complete Segal spaces. Then, we show that our model category for $T$-local complete Segal spaces remains compatible with the cartesian closure of simplicial spaces.

\begin{definition}\label{defmlocalsecatdef} Let $T$ be a set of simplicial maps. A complete Segal space $W$ is \emph{$T$-local} if, for any two objects $x$ and $y$ of $W$, their mapping space $\operatorname{map}_{W}(x, y)$ is $T$-local. A complete Segal space $W$ is \emph{rational} if, for any two objects $x$ and $y$ of $W$, their mapping space $\operatorname{map}_{W}(x, y)$ is rational.
\end{definition}
As complete Segal spaces rarely have a discrete space of $0$-simplices, we have no analog of Proposition \ref{disjointoffibers}. We use Theorem \ref{quilleneqinfty1cor} instead to replace a complete Segal space by a Segal category up to Dwyer-Kan equivalence.

\begin{proposition}\label{charlocalcsss} Let $T$ be a set of simplicial maps. Let $W$ be a complete Segal space. The following statements are equivalent.

\begin{enumerate}

\item The complete Segal space $W$ is $T$-local.

\item The Reedy fibrant Segal category $R(W)$ is $T$-local.

\item For every $i \geq 0$ and every $\lambda \in T$, the map $\operatorname{Map}_{\mathcal{SS}\mathrm{ets}^{\Delta^{\mathrm{op}}}}((\Delta[i]^{t} \times \lambda)_{r}, R(W))$ is a weak equivalence.

\item For every $i \geq 0$ and every $\lambda \in T$, the map $\operatorname{Map}_{\mathcal{SS}\mathrm{ets}^{\Delta^{\mathrm{op}}}}((\Delta[i]^{t} \times \lambda)_{r}, W)$ is a weak equivalence.

\end{enumerate}

\end{proposition}

\begin{proof} Firstly, (1) and (2) are equivalent by the Dwyer-Kan equivalence $R(W) \rightarrow W$ from Theorem \ref{quilleneqinfty1cor}. Then, (2) and (3) are equivalent by Proposition \ref{charlocalsecats}. Lastly, (3) and (4) are equivalent by adjointness.
\end{proof}
Thus, we localize Rezk's model category for complete Segal spaces $\mathcal{CSS}$ at the set $\widetilde{T}$ from Section \ref{section6}.

\begin{theorem}\label{localcssegalmodelcat} Let $T$ be a set of simplicial maps. There is a combinatorial, left proper, and simplicial model structure $\mathcal{L}_{T}\mathcal{CSS}$ on the category of simplicial spaces in which a map $f$ is

\begin{enumerate}

\item a weak equivalence if and only if, for every $T$-local complete Segal space $W$, the map $\operatorname{Map}_{\mathcal{SS}\mathrm{ets}^{\Delta^{\mathrm{op}}}}(f, W)$ is a weak equivalence; and

\item a cofibration if and only if $f$ is a monomorphism.

\end{enumerate}
A simplicial space $W$ is fibrant in $\mathcal{L}_{T}\mathcal{CSS}$ if and only if $W$ is a $T$-local complete Segal space.
\end{theorem}

\begin{corollary}\label{qlocalcssegalmodelcat} There is a model category whose fibrant objects are the rational complete Segal spaces.
\end{corollary}
We apply the criterion from Proposition \ref{rezkcriterionprop} to show that, for a set of simplicial maps $T$, our model category for $T$-local complete Segal spaces is cartesian. Our argument follows Rezk's proof of \cite[Lemma 10.3]{Rezk2001}, which is part of showing that the model category for complete Segal spaces is cartesian.

\begin{proposition}\label{biggestlemmayet} Let $T$ be a set of simplicial maps. If $W$ is a $T$-local complete Segal space, so is $W^{\Delta[1]^{t}}$.
\end{proposition}

\begin{proof} Replacing maps in $T$ by monomorphisms up to weak equivalence if necessary, we assume without loss of generality that $T$ comprises monomorphisms. By Theorem \ref{cssegalspacecartesian}, $W^{\Delta[1]^{t}}$ is a complete Segal space. For $i \geq 0$ and a map $A \rightarrow B$ in $T$, we show that
\[\operatorname{Map}_{\mathcal{SS}\mathrm{ets}^{\Delta^{\mathrm{op}}}}((\Delta[i]^{t} \times B)_{r}, W^{\Delta[1]^{t}}) \rightarrow \operatorname{Map}_{\mathcal{SS}\mathrm{ets}^{\Delta^{\mathrm{op}}}}((\Delta[i]^{t} \times A)_{r}, W^{\Delta[1]^{t}})\] is a weak equivalence. Equivalently, we show that the map
\[\operatorname{Map}_{\mathcal{SS}\mathrm{ets}^{\Delta^{\mathrm{op}}}}((\Delta[1]^{t} \times \Delta[i]^{t} \times B)_{r}, W) \rightarrow \operatorname{Map}_{\mathcal{SS}\mathrm{ets}^{\Delta^{\mathrm{op}}}}((\Delta[1]^{t} \times \Delta[i]^{t} \times A)_{r}, W)\] is a weak equivalence. For $0 \leq j \leq i$, we define the order-preserving map $\gamma^{j} \colon [i+1] \rightarrow [1] \times [i]$ by \[\gamma^{j}(z)=\begin{cases} (0, z), \textrm{ if } z \leq j; \\
(1, z-1), \textrm{ if } z > j.
\end{cases}\]
We also define the order-preserving map $\delta^{j} \colon [i] \rightarrow [1] \times [i]$ by \[\delta^{j}(z)=\begin{cases} (0, z), \textrm{ if } z \leq j; \\
(1, z), \textrm{ if } z > j.
\end{cases}\]
In this notation, the simplicial space $\Delta[1]^{t} \times \Delta[i]^{t}$ is the colimit of the diagram of inclusion maps

\[\begin{tikzcd}[row sep=3mm, column sep=3mm, ampersand replacement=\&]
	{\gamma^{0}(\Delta[i+1]^{t})} \&\& {\delta^{0}(\Delta[i]^{t})} \&\& {\gamma^{1}(\Delta[i+1]^{t})} \&\& {\delta^{1}(\Delta[i]^{t})} \&\& \cdots \&\& {\gamma^{i}(\Delta[i+1]^{t}).}
	\arrow[from=1-3, to=1-1]
	\arrow[from=1-3, to=1-5]
	\arrow[from=1-7, to=1-5]
	\arrow[from=1-7, to=1-9]
	\arrow[from=1-9, to=1-11]
\end{tikzcd}\]
For $0 \leq j \leq i$, because the set $T$ comprises monomorphisms, the inclusion maps \[(\gamma^{j}(\Delta[i+1]^{t}) \times A)_{r} \rightarrow (\gamma^{j}(\Delta[i+1]^{t}) \times B)_{r}\] and \[(\delta^{j}(\Delta[i]^{t}) \times A)_{r} \rightarrow (\delta^{j}(\Delta[i]^{t}) \times B)_{r}\] are acyclic cofibrations in our model category for $T$-local complete Segal spaces, so the induced map of colimits
\[(\Delta[1]^{t} \times \Delta[i]^{t} \times A)_{r} \rightarrow (\Delta[1]^{t} \times \Delta[i]^{t} \times B)_{r}\] is also an acyclic cofibration. As $W$ is a $T$-local complete Segal space, our claim follows.
\end{proof}

\begin{remark}\label{biggestlemmayetrmk} To relax the hypothesis that the set $T$ comprise monomorphisms, we would ultimately need the class of weak equivalences in a model category to be closed under colimits, which is not the case; a counterexample using pushouts of spaces is presented at the start of \cite[\S 10]{DwyerSpalinski1995} to motivate the use of homotopy pushouts.
\end{remark}

\begin{corollary}\label{allarecartesians} Let $T$ be a set of simplicial maps. The model category for $T$-local complete Segal spaces $\mathcal{L}_{T}\mathcal{CSS}$ is cartesian. In particular, the model category for rational complete Segal spaces is cartesian.
\end{corollary}

\section{Comparison of the two models for rational $(\infty, 1)$-categories}\label{section8}
Let $T$ be a set of simplicial maps. We show that the Quillen equivalence between complete Segal spaces and Segal categories descends to one between $T$-local complete Segal spaces and $T$-local Segal categories. Thus, we get a Quillen equivalence between rational complete Segal spaces and rational Segal categories, which captures that both structures are models for rational $(\infty, 1)$-categories. Lastly, we prove that our model category for Reedy fibrant $T$-local Segal categories can be left-induced from that for $T$-local complete Segal spaces.

\subsection{The equivalence of rational complete Segal spaces and rational Segal categories}\label{section8.1}
Let $T$ be a set of simplicial maps. We show that our model category for $T$-local complete Segal spaces $\mathcal{L}_{T}\mathcal{CSS}$ and our model category for Reedy fibrant $T$-local Segal categories $\mathcal{L}_{T}\mathcal{S}\mathrm{e}\mathcal{C}\mathrm{at}_{c}$ are Quillen equivalent. Our objective is to show that the Quillen equivalence $I \colon \mathcal{S}\mathrm{e}\mathcal{C}\mathrm{at}_{c} \rightleftarrows \mathcal{CSS} \colon R$ between Segal categories and complete Segal spaces from Theorem \ref{quilleneqinfty1} descends to a Quillen equivalence $I \colon \mathcal{L}_{T}\mathcal{S}\mathrm{e}\mathcal{C}\mathrm{at}_{c} \rightleftarrows \mathcal{L}_{T}\mathcal{CSS} \colon R$ between $T$-local Segal categories and $T$-local complete Segal spaces. To achieve such a descent, we need to left Bousfield localize both model categories in the Quillen equivalence from Theorem \ref{quilleneqinfty1} in a compatible fashion as follows.

\begin{theorem}[{\cite[Theorem 3.3.20(1)(b)]{Hirschhorn2003}}]\label{compatibleqeqs} Let $F \colon \mathcal{M} \rightleftarrows \mathcal{N} \colon G$ be a Quillen equivalence of combinatorial, left proper, and simplicial model categories. Suppose that every object of $\mathcal{M}$ is cofibrant. Let $T$ be a set of morphisms in $\mathcal{M}$. Then, $F \colon \mathcal{L}_{T}\mathcal{M} \rightleftarrows \mathcal{L}_{F(T)}\mathcal{N} \colon G$ is a Quillen equivalence.
\end{theorem}
Observe that every object of $\mathcal{S}\mathrm{e}\mathcal{C}\mathrm{at}_{c}$ is cofibrant. As the left adjoint in $I \colon \mathcal{S}\mathrm{e}\mathcal{C}\mathrm{at}_{c} \rightleftarrows \mathcal{CSS} \colon R$ is the inclusion and we localized both model categories at the same set, Theorem \ref{compatibleqeqs} yields our desired Quillen equivalence.

\begin{theorem}\label{theyareequivalentm} Let $T$ be a set of simplicial maps. The adjunction $I \colon \mathcal{L}_{T}\mathcal{S}\mathrm{e}\mathcal{C}\mathrm{at}_{c} \rightleftarrows \mathcal{L}_{T}\mathcal{CSS} \colon R$ is a Quillen equivalence.
\end{theorem}

\begin{corollary}\label{theyareequivalentq} The model category for rational complete Segal spaces and the model category for Reedy fibrant rational Segal categories are Quillen equivalent.
\end{corollary}

\subsection{An approach to rational Segal categories by way of left transfer}\label{section8.2}
Let $T$ be a set of simplicial maps. We show that, in analogy with our approach in Section \ref{section5.4}, the adjunction from Theorem \ref{theyareequivalentm} can be used to left-induce our model category for Reedy fibrant $T$-local Segal categories $\mathcal{L}_{T}\mathcal{S}\mathrm{e}\mathcal{C}\mathrm{at}_{c}$ from our model category for $T$-local complete Segal spaces $\mathcal{L}_{T}\mathcal{CSS}$. In particular, we left-induce our model category for Reedy fibrant rational Segal categories from our model category for rational complete Segal spaces.

Firstly, applying Theorem \ref{lefttransfertool} to the adjunction $I \colon \mathcal{S}\mathrm{e}\mathcal{C}\mathrm{at} \rightleftarrows \mathcal{L}_{T}\mathcal{CSS} \colon R$ yields the following model category.

\begin{proposition}\label{wellnotquite} Let $T$ be a set of simplicial maps. There is a combinatorial model structure on the category of Segal precategories in which a map $\psi$ is

\begin{enumerate}

\item a weak equivalence if and only if $\psi$ is a weak equivalence in $\mathcal{L}_{T}\mathcal{CSS}$; and

\item a cofibration if and only if $\psi$ is a monomorphism.

\end{enumerate}

\end{proposition}

\begin{proof} We verify the condition in Theorem \ref{lefttransfertool}. By our proof of Theorem \ref{vialeftransfer}, if a map $\psi$ of Segal precategories has the right lifting property against all monomorphisms of Segal precategories, then $\psi$ is a weak equivalence in $\mathcal{CSS}$, thus a weak equivalence in $\mathcal{L}_{T}\mathcal{CSS}$. Now, Theorem \ref{lefttransfertool} yields the model structure in our claim.
\end{proof}
It remains to verify that the model category from Proposition \ref{wellnotquite} is our model category for Reedy fibrant $T$-local Segal categories $\mathcal{L}_{T}\mathcal{S}\mathrm{e}\mathcal{C}\mathrm{at}_{c}$. In both model categories, the cofibrations are the monomorphisms, so it remains to prove that they also share the same weak equivalences. In other words, we need to show that a map of Segal precategories is a weak equivalence in $\mathcal{L}_{T}\mathcal{CSS}$ if and only if it is a weak equivalence in $\mathcal{L}_{T}\mathcal{S}\mathrm{e}\mathcal{C}\mathrm{at}_{c}$. Both implications rest on the fact that we can replace a complete Segal space by a Segal category and vice versa up to Dwyer-Kan equivalence. Our arguments are underpinned by the following lemma.

\begin{lemma}[{\cite[Corollary 9.3.3(1)]{Hirschhorn2003}}]\label{citeadnauseam2} Let $\mathcal{M}$ be a simplicial model category. If $X$ is a cofibrant object and $f$ is a weak equivalence of fibrant objects, then $\operatorname{Map}_{\mathcal{M}}(X, f)$ is a weak equivalence.
\end{lemma}
For the first implication, we invoke Theorem \ref{quilleneqinfty1cor} to replace a complete Segal space by a Segal category.

\begin{proposition}\label{thefirstimplication} Let $T$ be a set of simplicial maps. If a map of Segal precategories is a weak equivalence in $\mathcal{L}_{T}\mathcal{S}\mathrm{e}\mathcal{C}\mathrm{at}_{c}$, then it is a weak equivalence in $\mathcal{L}_{T}\mathcal{CSS}$.
\end{proposition}

\begin{proof} Let $\psi \colon A \rightarrow B$ be a map of Segal precategories that is a weak equivalence in $\mathcal{L}_{T}\mathcal{S}\mathrm{e}\mathcal{C}\mathrm{at}_{c}$. To prove that $\psi$ is a weak equivalence in $\mathcal{L}_{T}\mathcal{CSS}$, we show that, for every $T$-local complete Segal space $W$, the map \[\operatorname{Map}_{\mathcal{SS}\mathrm{ets}^{\Delta^{\mathrm{op}}}}(\psi, W) \colon \operatorname{Map}_{\mathcal{SS}\mathrm{ets}^{\Delta^{\mathrm{op}}}}(B, W) \rightarrow \operatorname{Map}_{\mathcal{SS}\mathrm{ets}^{\Delta^{\mathrm{op}}}}(A, W)\] is a weak equivalence. Combining Lemma \ref{citeadnauseam2} with the Dwyer-Kan equivalence $R(W) \rightarrow W$ from Theorem \ref{quilleneqinfty1cor}, we equivalently show that, for every $T$-local complete Segal space $W$, the map \[\operatorname{Map}_{\mathcal{SS}\mathrm{ets}^{\Delta^{\mathrm{op}}}}(\psi, R(W)) \colon \operatorname{Map}_{\mathcal{SS}\mathrm{ets}^{\Delta^{\mathrm{op}}}}(B, R(W)) \rightarrow \operatorname{Map}_{\mathcal{SS}\mathrm{ets}^{\Delta^{\mathrm{op}}}}(A, R(W))\] is a weak equivalence, which is the case because $R(W)$ is a Reedy fibrant $T$-local Segal category by Proposition \ref{charlocalcsss} and $\psi$ is a weak equivalence in $\mathcal{L}_{T}\mathcal{S}\mathrm{e}\mathcal{C}\mathrm{at}_{c}$. Hence, $\psi$ is a weak equivalence in $\mathcal{L}_{T}\mathcal{CSS}$.
\end{proof}
For the converse of Proposition \ref{thefirstimplication}, we recall the completion $\widehat{W}$ of a Segal space $W$ from Theorem \ref{completionfunctor}.

\begin{proposition}\label{theconverseof} Let $T$ be a set of simplicial maps. If a map of Segal precategories is a weak equivalence in $\mathcal{L}_{T}\mathcal{CSS}$, then it is a weak equivalence in $\mathcal{L}_{T}\mathcal{S}\mathrm{e}\mathcal{C}\mathrm{at}_{c}$.
\end{proposition}

\begin{proof} Let $\psi \colon A \rightarrow B$ be a map of Segal precategories that is a weak equivalence in $\mathcal{L}_{T}\mathcal{CSS}$. We need to show that, for every Reedy fibrant $T$-local Segal category $X$, the map \[\operatorname{Map}_{\mathcal{SS}\mathrm{ets}^{\Delta^{\mathrm{op}}}}(\psi, X) \colon \operatorname{Map}_{\mathcal{SS}\mathrm{ets}^{\Delta^{\mathrm{op}}}}(B, X) \rightarrow \operatorname{Map}_{\mathcal{SS}\mathrm{ets}^{\Delta^{\mathrm{op}}}}(A, X)\] is a weak equivalence. As a Reedy fibrant Segal category is a Segal space, we apply Lemma \ref{citeadnauseam2} to the Dwyer-Kan equivalence $X \rightarrow \widehat{X}$ and show that, for every Reedy fibrant $T$-local Segal category $X$, the map \[\operatorname{Map}_{\mathcal{SS}\mathrm{ets}^{\Delta^{\mathrm{op}}}}(\psi, \widehat{X}) \colon \operatorname{Map}_{\mathcal{SS}\mathrm{ets}^{\Delta^{\mathrm{op}}}}(B, \widehat{X}) \rightarrow \operatorname{Map}_{\mathcal{SS}\mathrm{ets}^{\Delta^{\mathrm{op}}}}(A, \widehat{X})\] is a weak equivalence, which is the case because $\widehat{X}$ is a $T$-local complete Segal space and $\psi$ is assumed to be a weak equivalence in $\mathcal{L}_{T}\mathcal{CSS}$. Therefore, $\psi$ is a weak equivalence in $\mathcal{L}_{T}\mathcal{S}\mathrm{e}\mathcal{C}\mathrm{at}_{c}$.
\end{proof}
Propositions \ref{thefirstimplication} and \ref{theconverseof} assemble to form the claim that we need to establish.

\begin{proposition}\label{desiredclaim} Let $T$ be a set of simplicial maps. A map of Segal precategories is a weak equivalence in $\mathcal{L}_{T}\mathcal{CSS}$ if and only if it is a weak equivalence in $\mathcal{L}_{T}\mathcal{S}\mathrm{e}\mathcal{C}\mathrm{at}_{c}$.
\end{proposition}
Proposition \ref{desiredclaim} yields our promised left transfer.

\begin{corollary}\label{vialefttransferm} Let $T$ be a set of simplicial maps. The model category $\mathcal{L}_{T}\mathcal{S}\mathrm{e}\mathcal{C}\mathrm{at}_{c}$ is left-induced from the model category $\mathcal{L}_{T}\mathcal{CSS}$ by means of the adjunction $I \colon \mathcal{S}\mathrm{e}\mathcal{C}\mathrm{at} \rightleftarrows \mathcal{L}_{T}\mathcal{CSS} \colon R$.
\end{corollary}
This left transfer gives a second description of our model category for Reedy fibrant rational Segal categories.

\bibliographystyle{abbrv}
\bibliography{MRI1C5}

\begin{thebibliography}{10}

\bibitem{Barwick2010}
C.~Barwick.
\newblock On left and right model categories and left and right {B}ousfield
  localizations.
\newblock {\em Homology Homotopy Appl.}, 12(2):245--320, 2010.

\bibitem{BHKKRS2015}
M.~Bayeh, K.~Hess, V.~Karpova, M.~K{\k{e}}dziorek, E.~Riehl, and B.~Shipley.
\newblock Left-induced model structures and diagram categories.
\newblock In {\em Women in Topology: Collaborations in Homotopy Theory}, volume
  641 of {\em Contemp. Math.}, pages 49--82. American Mathematical Society,
  2015.

\bibitem{Bergner2007_2}
J.~E. Bergner.
\newblock A characterization of fibrant {S}egal categories.
\newblock {\em Proc. Amer. Math. Soc.}, 135(12):4031--4037, 2007.

\bibitem{Bergner2007_3}
J.~E. Bergner.
\newblock A model category structure on the category of simplicial categories.
\newblock {\em Trans. Amer. Math. Soc.}, 359(5):2043--2058, 2007.

\bibitem{Bergner2007}
J.~E. Bergner.
\newblock Three models for the homotopy theory of homotopy theories.
\newblock {\em Topology}, 46(4):397--436, 2007.

\bibitem{Bergner2017}
J.~E. Bergner.
\newblock Equivalence of models for equivariant ({${\infty}$},1)-categories.
\newblock {\em Glasg. Math. J.}, 59(1):237--253, 2017.

\bibitem{Bergner2018}
J.~E. Bergner.
\newblock {\em The Homotopy Theory of ({${\infty}$},1)-Categories}.
\newblock Cambridge University Press, 2018.

\bibitem{BC2015}
J.~E. Bergner and S.~G. Chadwick.
\newblock Equivariant complete {S}egal spaces.
\newblock {\em Homology Homotopy Appl.}, 17(2):371--381, 2015.

\bibitem{BV1973}
J.~Boardman and R.~Vogt.
\newblock {\em Homotopy Invariant Algebraic Structures on Topological Spaces}.
\newblock Springer-Verlag Berlin Heidelberg, 1973.

\bibitem{Dugger2001}
D.~Dugger.
\newblock Combinatorial model categories have presentations.
\newblock {\em Adv. Math.}, 164(1):177--201, 2001.

\bibitem{DK1980}
W.~Dwyer and D.~Kan.
\newblock Function complexes in homotopical algebra.
\newblock {\em Topology}, 19(4):427--440, 1980.

\bibitem{DKS1989}
W.~Dwyer, D.~Kan, and J.~Smith.
\newblock Homotopy commutative diagrams and their realizations.
\newblock {\em J. Pure Appl. Algebra}, 57(1):5--24, 1989.

\bibitem{DwyerSpalinski1995}
W.~G. Dwyer and J.~Spalinski.
\newblock Homotopy theories and model categories.
\newblock In {\em Handbook of Algebraic Topology}, pages 1--56. Elsevier
  Science B.V., 1995.

\bibitem{FHT2001}
Y.~F{\'e}lix, S.~Halperin, and J.-C. Thomas.
\newblock {\em Rational Homotopy Theory}.
\newblock Springer-Verlag New York, 2001.

\bibitem{GoerssJardine1999}
P.~G. Goerss and J.~F. Jardine.
\newblock {\em Simplicial Homotopy Theory}.
\newblock Birkh{\"a}user Verlag, 1999.

\bibitem{GTHT2000}
A.~G{\'{o}}mez-Tato, S.~Halperin, and D.~Tanr{\'{e}}.
\newblock Rational homotopy theory for non-simply connected spaces.
\newblock {\em Trans. Amer. Math. Soc.}, 352(4):1493--1525, 2000.

\bibitem{Heuts2021}
G.~Heuts.
\newblock Lie algebras and {$v_{n}$}-periodic spaces.
\newblock {\em Ann. of Math.}, 193(1):223--301, 2021.

\bibitem{Hirschhorn2003}
P.~S. Hirschhorn.
\newblock {\em Model Categories and Their Localizations}.
\newblock American Mathematical Society, 2003.

\bibitem{HS2001}
A.~Hirschowitz and C.~Simpson.
\newblock Descente pour les n-champs ({D}escent for n-stacks).
\newblock Available at arXiv:math/9807049v3 [math.AG].

\bibitem{Joyal2002}
A.~Joyal.
\newblock Quasi-categories and {K}an complexes.
\newblock {\em J. Pure Appl. Algebra}, 175(1):207--222, 2002.

\bibitem{Lurie2009}
J.~Lurie.
\newblock {\em Higher Topos Theory}.
\newblock Princeton University Press, 2009.

\bibitem{Pellissier2003}
R.~Pellissier.
\newblock Weak enriched categories - categories enrichies faibles.
\newblock Available at arXiv:math/0308246v1 [math.AT].

\bibitem{Quillen1967}
D.~G. Quillen.
\newblock {\em Homotopical Algebra}.
\newblock Springer-Verlag Berlin Heidelberg, 1967.

\bibitem{Reedy1974}
C.~Reedy.
\newblock Homotopy theory of model categories.
\newblock Available at \url{https://math.mit.edu/~psh/reedy.pdf}.

\bibitem{Rezk2001}
C.~Rezk.
\newblock A model for the homotopy theory of homotopy theory.
\newblock {\em Trans. Amer. Math. Soc.}, 353(3):973--1007, 2001.

\bibitem{Rezk2010}
C.~Rezk.
\newblock A cartesian presentation of weak {$n$}-categories.
\newblock {\em Geom. Topol.}, 14(1):521--571, 2010.

\end{thebibliography}

\end{document}